\documentclass[10pt,amsfonts, epsfig]{amsart}
\usepackage{amsmath, amscd, amssymb}
\usepackage{graphpap, color}
\usepackage[mathscr]{eucal}
\usepackage{mathrsfs}
\usepackage{pstricks}
\usepackage{color}
\usepackage{cancel}
\usepackage[mathscr]{eucal}
\usepackage{stmaryrd}

\usepackage[pagebackref]{hyperref}

\usepackage[T1]{fontenc}
\usepackage{textcomp}

\usepackage{cite}
\usepackage[all,cmtip]{xy}

\numberwithin{equation}{section}

\newcommand{\CC}{\mathbb{C}}

\newcommand{\RR}{\mathbb{R}}
\newcommand{\ZZ}{\mathbb{Z}}

\def\cln{\colon}

\newcommand{\cal}{\mathcal}

\def\cK{{\cal K}}
\def\cL{{\cal L}}

\def\cX{{\cal X}}
\def\cY{{\cal Y}}

\def\begeq{\begin{equation}}
\def\endeq{\end{equation}}
\def\and{\quad{\rm and}\quad}

\def\and{\quad\text{and}\quad}

\DeclareMathOperator{\diff}{{Diff}}

\newtheorem{prop}{Proposition}[section]
\newtheorem{theo}[prop]{Theorem}
\newtheorem{lemm}[prop]{Lemma}

\newtheorem{defi}[prop]{Definition}
\newtheorem{conj}[prop]{Conjecture}

\newtheorem{defi-prop}[prop]{Definition-Proposition}
\newtheorem{quest}[prop]{Question}

\def\beq{\begin{equation}}
\def\eeq{\end{equation}}

\def\bee{\begin{equation}}
\def\eeq{\end{equation}}

\begin{document}
\title{Local isomorphisms for families of projective non-unruled manifolds}

\author{Mu-lin Li}
\address{School of Mathematics, Hunan University, China}
\email{mulin@hnu.edu.cn}

\thanks{The author is supported by NSFC (No. 12271073 and 12271412).}

\date{\today}

\maketitle

\begin{abstract}
Let $\pi\cln \cX\to S$ and  $\pi\cln \cY\to S$ be two smooth families of projective non-uniruled manifolds over a Riemann surface $S$ (probably non-compact). Suppose these two families are pointwise isomorphic. We prove that  there exists an open dense subset $U\subset S$ such that the two restricted families are locally isomorphic over $U$.  This partially answers Wehler's question on locally isomorphic of families of compact complex manifolds.
\end{abstract}

\section{introduction}

Let $\pi:\cX\to S$ be a proper holomorphic submersion between complex manifolds $\cX$ and $S$. Therefore all the fibers of $\pi$ are compact complex manifolds of equal dimension. Such a proper submersion is called a smooth family of compact complex manifolds. Let $\pi:\cX\to S$ and $\pi^{\prime}:\cY\to S$ be two families of compact complex manifolds over a complex manifold $S$.  We say these
families are pointwise isomorphic if there exists a bi-holomorphism $f_s:\cY_s=\pi^{\prime-1}(s)\to \cX_s:=\pi^{-1}(s)$ for every point $s\in S$.

We say the two families $\pi$ and $\pi'$ are locally isomorphic at $s$ if there exists an open neighborhood $U_s\subset S$ and a bi-holomorphism $g_{U_s}:\cY_{U_s}\to \cX_{U_s}$ such that $s\in U_s$ and $\pi_{U_s}'=\pi_{U_s}\circ g_{U_s}$. Here $ \pi_{U_s}: \cX_{U_s}=\pi^{-1}({U_s})\to {U_s}$ is the restriction of $\pi$, and $ \pi'_{U_s}: \cY_{U_s}=\pi^{\prime-1}({U_s})\to {U_s}$ is the restriction of $\pi'$.

In 1977, Wehler posed the following question.
\begin{quest}Let $\pi:\cX\to S$ and $\pi':\cY\to S$ be two smooth families of compact complex manifolds over a complex manifold $S$. Suppose that these two families are pointwise isomorphic. If the function
\beq
s\to \dim H^0(\cY_s,T_{\cY_s})\nonumber
\eeq
is constant on $S$, are these two families locally isomorphic at every point of $S$?
\end{quest}
Wehler \cite{We} proved that it is true for families of complex tori and families  of compact manifolds with negative curved holomorphic curvature.  Meersseman \cite{Mee} attempted to prove it for arbitrary families of compact manifolds, but a counterexample constructed by Kirschner \cite{Kir} revealed gaps in his proof. This led to the proposal of the following weakened version of the question.
\begin{conj}\label{conj-1}Let $\pi:\cX\to S$ and $\pi':\cY\to S$ be two smooth families of compact complex manifolds over a complex manifold $S$. Suppose that these two families are pointwise isomorphic. If the function
\beq
s\to \dim H^0(\cY_s,T_{\cY_s})\nonumber
\eeq
is constant on $S$, then there exists an open dense subset $U\subset S$ such that the restricted families $\pi_{U}: \cX_{U}=\pi^{-1}({U})\to {U}$ and $\pi'_{U}: \cY_{U}=\pi^{\prime-1}({U})\to {U}$ are locally isomorphic at every point $s\in U$.
\end{conj}
In this paper, we focus on families of projective non-uniruled manifolds. By using the coarse moduli space of polarized K\"ahler manifolds constructed by Fujiki\cite{Fu84} and Schumacher \cite{Sch84}, we establish the following theorem.
\begin{theo}
Let $\pi:\cX\to S$ and $\pi':\cY\to S$ be two smooth families of projective non-uniruled manifolds over a Riemann surface $S$. Suppose the two families are pointwise isomorphic. Then there exists an open dense subset $U\subset S$ such that the restricted families $\pi_{U}: \cX_{U}=\pi^{-1}({U})\to {U}$ and $\pi'_{U}: \cY_{U}=\pi^{\prime-1}({U})\to {U}$ are locally isomorphic at every point $s\in U$.
\end{theo}

The main idea of the proof  is as following. For two smooth families of projective non-uniruled manifolds $\pi:\cX\to \Delta$ and $\pi':\cY\to\Delta$, we first construct two new families of polarized  projective non-uniruled manifolds $(\pi_U:\cX_U\to U,\alpha)$ and $(\pi'_U:\cY_U\to U, \beta)$. These two families induce two holomorphic maps $f,\, f': U\to M$ satisfying $f(U)=f'(U)$, where $M$ is the coarse moduli of polarized  projective non-uniruled manifolds.  Then by using the fact that the moduli functor of  polarized  projective non-uniruled manifolds is locally complete, we finish the proof.


\section{Preliminaries:Diffeomorphism groups and Moduli space of polarized K\"ahler manifolds}
\subsection{Diffeomorphism groups}
First, we recall some basic properties of the space of smooth maps. Let $X$ and $Y$ be two compact complex manifolds. Denote by $C^{\infty}(Y,X)$ the space of smooth maps between $Y$ and $X$, endowed with the compact open $C^{\infty}$ topology. Let $\diff(Y,X)$ be the space of diffeomorphisms between $Y$ and $X$ (assuming that they are diffeomorphic to each other). We have the following properties.
\begin{prop}\cite[2.6 Lemma]{Schm}$\diff(Y,X)$ is an open subset of $C^{\infty}(Y,X)$ in compact open $C^{\infty}$ topology.
\end{prop}
\begin{prop}\cite[Proposition 7]{Lesl}\label{homotopy-1}Let $X,Y,Z$ be compact complex manifolds, then the composition
\beq
C^{\infty}(Z,Y)\times C^{\infty}(Y,X)\to C^{\infty}(Z,X)
\eeq
is smooth in compact open $C^{\infty}$ topology.
\end{prop}
\begin{prop}The space $\diff(Y,X)$ has only countable many connected components.
\end{prop}
\begin{proof}Let $\diff(Y):=\diff(Y,Y)$ be the diffeomorphism group of $Y$. It has only countable many connected components because it is locally arc-connected and  has the homotopy
type of a countable CW complex by \cite[Section 1.1.2]{ABK}. By Proposition \ref{homotopy-1}, the space $\diff(Y,X)$ is isomorphic to $\diff(Y)$. Thus $\diff(Y,X)$ has only countable many connected components.
\end{proof}
Then we have the following equality on induced morphisms on cohomologies groups.
\begin{prop}\label{equal-1}Let $f,g:Y\to X$ be two diffeomorphisms in the same connected component of $\diff(Y,X)$. Then the induced morphisms on cohomology groups are equal
\beq
f^*=g^*: H^2(X,R)\to H^2(Y, R),\nonumber
\eeq
where $R=\ZZ,\RR$ or $\CC$.
\end{prop}
\subsection{Moduli space of polarized K\"ahler manifolds}

\begin{defi}A polarized compact K\"ahler manifold is a pair $(X, \omega_{X})$ consisting of a connected compact complex manifold $X$ and a K\"ahler class $\omega_{X}\in H^2(X,\RR)$.
\end{defi}

An isomorphism of two polarized compact K\"ahler manifolds $(X, \omega_{X})$ and $(\widetilde{X},\omega_{\widetilde{X}})$ is an holomorphic isomorphism $\phi:X\to \widetilde{X}$ with $\phi^*(\omega_{\widetilde{X}})=\omega_{X}\in H^2(X,\RR)$.

\begin{defi}A polarized family of compact K\"ahler manifolds (parametrized by a complex space $B$) is a pair $(f,\omega_{\cX/B})$ consisting of a proper smooth morphism $f:\cX\to B$ of complex spaces with connected fibers and an element
 $\omega_{\cX/B}\in \Gamma(B,R^2f_*\RR)$ such that
 \begin{enumerate}
 \item The value of the section $\omega_{\cX_b}=(\omega_{\cX/B})(b)$ on $b$ is a K\"ahler class on each fiber $\cX_b =f^{-1}(b)$.
  \item $\eta(\omega_{\cX/B})=0$ where $\eta:\Gamma(B,R^2f_*\RR)\to\Gamma(B,R^2f_*\mathcal{O}_{\cX})$ is the homomorphism induced by the natural inclusion $\RR\to \mathcal{O}_{\cX}$ of sheaves on $\mathcal{X}$.
 \end{enumerate}
\end{defi}

%
%
Let $\mathcal{M}$ be the groupoid fibered category over the category of reduced complex spaces, such that for a reduced complex space $B$ the groupoid $\mathcal{M}(B)$ is as follows
 \begin{align*}
 \mathcal{M}(B)=&\bigg\{(f: \cX\to B,\omega_{\cX/B})\, \mbox{is a family of polarized K\"ahler manifolds };\\
&(\cX_b,\omega_{\cX_b})\, \mbox{is a polarized non-uniruled K\"ahler manifold,\,  $\forall b\in B$}\bigg\}.
\end{align*}

 An arrow of $\mathcal{M}(B)$ from $(f: \cX\to B,\omega_{\cX/B})$ to $(f': \cX'\to B,\omega_{\cX'/B})$ is an isomorphism $\Phi:  \cX\to \cX'$ such that $\Phi^*(\omega_{\cX'/B})=\omega_{\cX/B}$.

\begin{defi} Let $\left(f:\cX\to S,\omega_{\cX}\right)$ be a polarized family of compact K\"ahler manifolds. Let $s\in S$. Then we say that $\left(f:\cX\to S,\omega_{\cX}\right)$ is locally complete at $s$ if for
any polarized deformation $\left(f':\cX'\to S',\omega_{\cX'}\right)$ of $(\cX_s,\omega_{\cX_s})$, there exists a morphism $\tau: S'\to S$, such that $\left(f':\cX'\to S',\omega_{\cX'}\right)$ is isomorphic
to the pull-back of $\left(f:\cX\to S,\omega_{\cX}\right)$ to $S'$. If, further, $\tau$ is unique for any $\left(f':\cX'\to S',\omega_{\cX'}\right)$ as above, we say that $\left(f:\cX\to S,\omega_{\cX}\right)$, considered as a polarized deformation of $(\cX_s,\omega_{\cX_s})$, is the local modular family of $(\cX_s,\omega_{\cX_s})$. In this case $\tau$ is called the universal map associated to $\left(f':\cX'\to S',\omega_{\cX'}\right)$.
\end{defi}
Such a local modular family $\left(f:\cX\to S,\omega_{\cX}\right)$ is  usually called the Kuranish family of $(\cX_s,\omega_{X_s})$. The following Proposition asserts the existence of the Kuranishi family for polarized non-uniruled manifolds.

\begin{prop}\cite[Proposition 8]{Fu84}\label{moduli-1} For any polarized compact K\"ahler manifold $(X,\omega_X)$ such that $Aut_0(X)$ is a complex torus its local modular family $\left(f:\cX\to S,\omega_{\cX}\right)$ with $(\cX_0,\omega_{\cX_0})=(X,\omega_X)$ exists. Moreover it is locally complete at every point in a neighborhood of $0$.
\end{prop}

By \cite[Section 1, Theorem]{Fu84}, \cite[(1.3)Theorem]{Sch84} and \cite[Theorem 5.9]{FS90}, there exists a coarse separated moduli complex space $M$ for $\mathcal{M}$.

\section{Proof of the main theorem}
When $\pi\cln \cX\to \Delta$ is a smooth family of compact complex manifolds, by \cite[Theorem 9.3]{Vo} $\pi:\cX\to\Delta$ forms a smooth fiber bundle. There exists a diffeomorphism
\beq
T:\cX\cong \cX_{t_1}\times \Delta
\eeq
such that $pr_2\circ T=\pi$, where $pr_2:\cX_{t_1}\times \Delta\to \Delta$, and $T|_{\pi^{-1}(t_1)}=Id:\cX_{t_1}\to \cX_{t_1}$.

   The explanation of the existence of $T$ is from \cite{Math-1}: For constant map $F_0:\Delta\to \Delta$ which maps the disk to $t_1$, Let $F:\Delta\times [0,1]\to \Delta$ be a smooth homotopy from  $F_0:\Delta\to \Delta$ to identity map $F_1:\Delta\to \Delta$. The pullback $F^*\cX$ is a smooth fiber bundle over $\Delta\times [0,1]$. By \cite[Theorem 5]{del}, this fiber bundle admits a complete Ehresmann connection.  Let  $v$ be the trivial vector field of $[0,1]$, it induces a vector field $\widetilde{v}$ on $\Delta\times [0,1]$. Denote by $\mathcal{V}$ the horizontal vector field over $F^*\cX$ induced by pullback of $\widetilde{v}$. Then the flow induced by $\mathcal{V}$ yields the isomorphism between $F_0^*\cX$ to $F^*_1\cX$, which gives the isomorphism $T^{-1}: \cX_{t_1}\times \Delta \to \cX$.

   Let  $T_{t_1}:=pr_1\circ T: \cX\to \cX_{t_1}$ and $\kappa_t: \cX_t\to \cX_{t_1}$ be the composition of $\imath_t: \cX_t\to \cX$ with $T_{t_1}$. The diffeomorphisms $\{\kappa_t\}$
 depend smoothly on $t$. Let $\psi_t:\cX_{t_1}\to \cX_t$ be the inverse diffeomorphism of $\kappa_t$. See the following commutative diagram.
\begin{eqnarray*}
\xymatrix{
\cX_t \ar[r]^{\iota_t} \ar@<.5ex>[rd]^{\kappa_t}  &  \cX\ar[d]^{T_{t_1}}\ar[r]^(0.35){T}\ar[rd]^(0.35){\pi}|\hole  & \cX_{t_1}\times\Delta \ar[dl]^(0.6){p_1}\ar[d]^{p_2}\\
& \cX_{t_1}\ar@<.5ex>[lu]^{\psi_t}  & \Delta
}
\end{eqnarray*}

Because $\pi:\mathcal{X}\to \Delta$ is a smooth family of compact complex manifolds, we have that
\beq\label{iden-1}
\psi_t^*H^{k}(\cX_t,\ZZ)=H^{k}(\cX_{t_1},\ZZ),\, \psi_t^*H^{k}(\cX_t,\CC)=H^{k}(\cX_{t_1},\CC).
\eeq
Denote by $H^{p,q}(\cX_t,\CC)$ the Dolbeault cohomology groups of the projective manifold $\cX_t$. We have the following Hodge decomposition
\beq
H^2(\cX_t,\CC)=H^{2,0}(\cX_t,\CC)\oplus H^{1,1}(\cX_t,\CC)\oplus H^{0,2}(\cX_t,\CC).\nonumber
\eeq
\begin{lemm}\label{Key-lemm-1}Let $\pi\cln \cX\to \Delta$ be a smooth family of projective manifolds. For an element $\alpha\in H^2(\cX,\ZZ)$, suppose that there exists an uncountable  subset $E\subset \Delta$ such that the restriction $a_t:=\iota_t^*\alpha\in H^{1,1}(\cX_t,\CC)$ for $t\in E$. Then there exists a line bundle $\cL$ over $\cX$ such that $c_1(\cL)=\alpha$.
\end{lemm}
\begin{proof}

 By the exponential sequence, one has the commutative diagram of exact sequences
\begin{equation*}\label{lesa}
\xymatrix@C=0.5cm{
  \cdots \ar[r]^{}
  & H^1(\mathcal{X}, \mathcal{O}^{*}_{\mathcal{X}}) \ar[r]^{}
  & H^2(\mathcal{X}, \mathbb{Z}) \ar[d]_{} \ar[r]^{\varphi}\ar[d]^{\cong}
  & H^2(\mathcal{X}, \mathcal{O}_{\mathcal{X}})\ar[d]^{\cong}\ar[r]^{} & \cdots \\
    &  & \Gamma(\Delta, R^2\pi_*\ZZ) \ar[r]^{\widetilde{\varphi}}
  &\Gamma(\Delta, R^2\pi_*\mathcal{O}_{X_{t}})&}.
\end{equation*}
Let $s_{\alpha}$ be the section in $\Gamma(\Delta,R^2\pi_*\ZZ)$ induced by $\alpha$, thus $s_{\alpha}(t)=a_t\in H^2(\cX_t,\ZZ)$. Define the zero locus
\beq
Z(s_{\alpha}):=\{t\in\Delta|a_t \mbox{ is the first Chern class of a line bundle on }\cX_t\}.\nonumber
 \eeq
 It is an analytic subset of $\Delta$. By Lefschetz theorem on $(1,1)$ classes, we have $E\subset Z(s_{\alpha})$. Since $E$ is uncountable, $Z(s_{\alpha})$ can't be a proper analytic subset of $\Delta$. Thus $Z(s_{\alpha})=\Delta$. Therefore $\widetilde{\varphi}(s_\alpha)=0$, which implies that $\varphi(\alpha)=0$ by the above commutative diagram. By the exponent sequence there exists a line bundle $\cL$ over $\cX$ such that $c_1(\cL)=\alpha$.
\end{proof}

\begin{lemm}\label{Key-lemm-2}Let $\pi\cln \cX\to \Delta$ be a smooth family of projective manifolds. Suppose that there exists an uncountable  subset $E\subset \Delta$ such that the class $a_t\in H^{1,1}(\cX_t,\CC)$ is an  integral class for all $t\in E$. Then there exists a point $t_0\in E$ and a line bundle $\cL$ over $\cX$ such that $\iota_{t_0}^*\left(c_1(\cL)\right)=a_{t_0}$.
\end{lemm}
\begin{proof}
For $t\in E$, let $\widetilde{a}_t=\psi_s^*a_{t}\in H^{2}(\cX_{t_1},\ZZ)$, and
\beq
\alpha_t:=T_{t_1}^*{\widetilde{a}_t}\in H^2(\mathcal{X}, \mathbb{Z}).\nonumber
\eeq
then the restriction $a_t=\iota_t^*\alpha_t\in H^{1,1}(\cX_t,\CC)$. Let $s_{\alpha_t}$ be the section in $\Gamma(\Delta,R^2\pi_*\ZZ)$ induced by $\alpha_t$, then $s_{\alpha_t}(t)=a_t\in H^{1,1}(\cX_t,\CC)$ for $t\in E$. Therefore
\beq
E\subset \cup_{\alpha_t \in H^2(\mathcal{X}, \mathbb{Z})}Z(s_{\alpha_t}).\nonumber
\eeq
Because there are only countable many elements in $H^2(\cX,\ZZ)$, and $E$ is an uncountable set, there exists a point $t_0\in E$ such that $Z(s_{\alpha_{t_0}})$ contains uncountable many points. By Lemma \ref{Key-lemm-1}, there exists a line bundle $\cL$ over $\cX$ with $\iota_{t_0}^*\left(c_1(\cL)\right)=a_{t_0}$.

\end{proof}

\begin{lemm}\label{equal-2}Let $U$ be a holomorphic convex domain in $\CC$, and $f,g:U\to N$ be two holomorphic maps from $U$ to a complex space $N$. Suppose that there exists an uncountable set $E\subset U$ such that $f(t)=g(t)$ for all $t\in E$. Then $f=g$ on the whole $U$.
\end{lemm}
\begin{proof}Let $t_0\in U$ be an accumulation point of $E$, then there exists a small disk $\Delta_{t_0}$ around $t_0$ such that $f(t)=g(t)$ for $t\in\Delta_{t_0}$ since $f,g$ are holomorphic maps. Since $U$ is holomorphic convex domain, arbitrary two points can be connected by a smooth curve. Fix a point $q\in \Delta_{t_{0}}$, and let $p$ be an arbitrary point in $U\setminus \Delta_{t_{0}}$. Let $\gamma:[0,1]\rightarrow U$ be a smooth path connecting $\gamma(0)=q$ and $\gamma(1)=p$. Let $a_{0}=\text{min}\{u\in[0,1];\gamma(u)\in
\partial\Delta_{t_{0}}\}$, so that $\gamma(a_0)\in \partial\Delta_{t_{0}}$ and $\gamma([0,a_0))\subset\Delta_{t_0}$.
Define
$$A:=\{u\in [0,1]|f(\gamma(u))=g(\gamma(u))\}.$$
Since $\gamma([0,a_0))\subset\Delta_{t_0}$, we have $[0,a_0)\subseteq A$. Because $f,g$ are holomorphic maps, $f(\gamma(a_0))=g(\gamma(a_0))$, and there exists small neighborhood $(a_0-\epsilon,a_0+\epsilon)$ of $a_0$ such that $f(\gamma(s))=g(\gamma(s))$ for $s\in (a_0-\epsilon,a_0+\epsilon)$. Thus the set $A$ is open and closed. So $A=[0,1]$. Therefore $f=g$ on the whole domain $U$.
\end{proof}

  Let  $\widetilde{T}_{t_1}:=pr_1\circ \widetilde{T}: \cY\to \cY_{t_1}$ and $\widetilde{\kappa}_t: \cY_t\to \cY_{t_1}$ be the composition of $\widetilde{\iota}_t: \cY_t\to \cY$ with $\widetilde{T}_{t_1}$. The diffeomorphisms $\{\widetilde{\kappa}_t\}$
are smooth depend on $t$. Let $\widetilde{\psi}_t:\cY_{t_1}\to \cY_t$ be the inverse diffeomorphism of $\widetilde{\kappa}_t$. See the following commutative diagram.

\begin{eqnarray*}
\xymatrix{
\cY_t \ar[r]^{\widetilde{\iota}_t} \ar@<.5ex>[rd]^{\widetilde{\kappa}_t}  &  \cY\ar[d]^{\widetilde{T}_{t_1}}\ar[r]^(0.35){\widetilde{T}}\ar[rd]^(0.35){\pi'}|\hole  & \cY_{t_1}\times\Delta \ar[dl]^(0.6){\widetilde{p}_1}\ar[d]^{\widetilde{p}_2}\\
& \cY_{t_1}\ar@<.5ex>[lu]^{\widetilde{\psi}_t}  & \Delta
}
\end{eqnarray*}

\begin{prop}\label{main-1}
Let $\pi:\cX\to \Delta$ be a smooth family of projective non-unruled manifolds over unit disk $\Delta$. Suppose that $\pi':\cY\to \Delta$ is a smooth family of compact complex manifolds satisfies that $f_t:\cY_t\to \cX_t$ is isomorphic pointwise. Then there exists an open dense subset $U\subset \Delta$ such that for $t\in U$, the families are local isomorphic.
\end{prop}
\begin{proof}{\bf Step 1.} In this step we construct a line on $\cX$ which is relatively ample over an open dense subset $U'\subset \Delta$.

Because $\cX_t$ are projective manifolds for all $t\in\Delta$, choosing  an ample integral class $a_t\in H^{2}(\cX_t,\CC)$  for every $t\in\Delta$, Denote by $\bar{a}_t:=\psi_t^*a_t\in H^2(\cX_{t_1},\ZZ)$. Let $\alpha_t:=T_{t_1}^*\bar{a}_t$, then
\beq
\cup_{t\in \Delta}Z(s_{\alpha_t})=\Delta.\nonumber
\eeq
By Lemma \ref{Key-lemm-2}, there exists a $t_0\in\Delta$ such that
\beq
Z(s_{\alpha_{t_0}})=\Delta\nonumber
\eeq
and a line bundle $\cL$ on $\cX$ such that $\iota_{t_0}^*\left(c_1(\cL)\right)=\iota_{t_0}^*\left(\alpha_{t_0}\right)=a_{t_0}$. Without losing of generality we can chose $t_1$ to be $t_0$.  Because  $\iota_{t_1}^*(\alpha_{t_1})=a_{t_1}$ is an ample class, there exists an Zariski open subset $U'\subset \Delta$ such that $\cL_t:=\cL|_{\cX_t}$ is ample for every $t\in U'$, thus $b_t=c_1(\cL_t)=\iota_{t}^*(\alpha_{t_1})$ is an  ample class for every $t\in U'$.

{\bf Step 2.} Constructing a line bundle on $\cY$ which is relatively ample over a Zariski open dense subset.

 Consider the compositions of diffeomorphisms
\beq
\kappa_t\circ f_t\circ\widetilde{\psi}_t:\cY_{t_1}\to \cY_t\to\cX_t\to\cX_{t_1}\nonumber
\eeq
for all $t\in\Delta$, since $\diff(\cY_{t_1},\cX_{t_1})$ has only countable many components, there exists an uncountable subset $\widetilde{E}\subset \Delta$ such that $\kappa_t\circ f_t\circ\widetilde{\psi}_t$ belongs to the same component of $\diff(\cY_{t_1},\cX_{t_1})$ for $t\in \widetilde{E}$. Thus by Proposition \ref{equal-1}, we have
\beq\label{equal-4}
\widetilde{\psi}_t^*\circ f_t^*\circ\kappa_t^*=\widetilde{\psi}_{\widetilde{t}}^*\circ f_{\widetilde{t}}^*\circ\kappa_{\widetilde{t}}^*
\eeq
for $\widetilde{t},t\in \widetilde{E}$.

For $t\in \widetilde{E}$, let $c_t:=f_t^*b_t\in H^2(\cY_t,\ZZ)\cap H^{1,1}(\cY_t,\CC)$. Define $\widetilde{c}_t:=\widetilde{\psi}_t^*(c_t)$ and $\beta_t:=\widetilde{T}_{t_1}^*\widetilde{c}_t\in H^2(\cY,\ZZ)$. By Lemma \ref{Key-lemm-2}, there exists a point  $\widetilde{t}_1\in \widetilde{E}$, such that
\beq
Z\left(s_{\beta_{\widetilde{t}_1}}\right)=\Delta.\nonumber
\eeq
Thus there exists a line bundle $\widetilde{\cL}$ on $\cY$ such that the first Chern class $c_1(\widetilde{\cL})=\beta_{\widetilde{t}_1}$. Because $c_{\widetilde{t_0}}$ is an ample class, there exists an Zariski open subset $\widetilde{U}\subset \Delta$ such that $\widetilde{\iota}^*_t(\beta_{\widetilde{t}_1})$ is an ample class on $\cY_t$ for every $t\in \widetilde{U}$.

{\bf Step 3.} Constructing two smooth families of polarized K\"ahler manifolds which isomorphic on an uncountable subset.

Let $U:=U'\cap \widetilde{U}$, which is also a Zariski open dense subset of $\Delta$. Thus $E:=\widetilde{E}\cap U$ is an uncountable subset.
Define $\widetilde{\rho}_t:=\widetilde{\psi}_{\widetilde{t}_1}\circ\widetilde{\kappa}_{t}:\cY_t\to \cY_{\widetilde{t}_1}$, $\widetilde{\varphi}_t:=\widetilde{\psi}_{t}\circ\widetilde{\kappa}_{t_1}:\cY_{\widetilde{t}_1}\to \cY_t$,
$\rho_t:=\psi_{\widetilde{t}_1}\circ\kappa_{t}:\cX_t\to \cX_{\widetilde{t}_1}$, and $\varphi_t:=\psi_{t}\circ\kappa_{\widetilde{t}_1}:\cX_{\widetilde{t}_1}\to \cX_t$. Then
\beq
b_{\widetilde{t}_1}=\varphi_t^*(b_t).
\eeq
By the equality of (\ref{equal-4}), for $t\in E$,
\begin{eqnarray*}
\widetilde{\iota}_t^*(\beta_{\widetilde{t}_1})&=&\widetilde{\iota}_t^*\widetilde{T}_{t_1}^*\widetilde{c}_{\widetilde{t}_1}\\
&=&\widetilde{\iota}_t^*\widetilde{T}_{t_1}^*\widetilde{\psi}_{\widetilde{t}_1}^*(c_{\widetilde{t}_1})\\
&=&\widetilde{\iota}_t^*\widetilde{T}_{t_1}^*\widetilde{\psi}_{\widetilde{t}_1}^*(f_{\widetilde{t}_1}^*b_{\widetilde{t}_1})\\
&=&\widetilde{\iota}_t^*\widetilde{T}_{t_1}^*\widetilde{\psi}_{\widetilde{t}_1}^*\left(f_{\widetilde{t}_1}^*(\varphi_t^*(b_t))\right)\\
&=&\widetilde{\iota}_t^*\widetilde{T}_{t_1}^*\widetilde{\psi}_{\widetilde{t}_1}^*\left(f_{\widetilde{t}_1}^*(\kappa_{\widetilde{t}_1}^*\psi_t^*(b_t))\right)\\
&=&\widetilde{\iota}_t^*\widetilde{T}_{t_1}^*\widetilde{\psi}_t^*\circ f_t^*\circ\kappa_t^*\psi_t^*(b_t)\\
&=&f_t^*(b_t).
\end{eqnarray*}

Thus $(\cY_t,\widetilde{\iota}_t^*(\beta_{\widetilde{t}_1}))\cong (\cX_t, b_t)$ via $f_t$ for every $t\in E$. Recall that $M$ is the coarse moduli space of polarized K\"ahler manifolds. Families $\left(\pi_u^{\prime}:\cY_U\to U, \beta_{\widetilde{t}_1}\right)$ and $\left(\pi_u:\cX_U\to U, \alpha_{t_1}\right)$ induces two analytic map $g: U\to M$ and $h:U\to M$ which equal at all points $t\in E$. Therefore
\beq\label{equal-3}
g(t)=h(t),
\eeq
for all $t\in U$ by Lemma \ref{equal-2}.

{\bf Step 4.} Using the completeness of the moduli functor of polarized non-uniruled manifolds to prove the result.

For a point $t\in U$, let $(\pi_{K_t}:\cK\to K_t,\omega_{\cK})$ be its Kuranish family such that $(\cK_0,\omega_{\cK}|_{\cK_0})\cong \left(\cY_t, \beta_t\right)$.
By Proposition \ref{moduli-1}, there exists a small neighborhood $U_t\subset U$, and morphisms $\widetilde{g},\widetilde{h}: U_t\to K_t$ such that
\beq
\cY_{U_t}\cong \widetilde{g}^*\cK,\nonumber
\eeq
\beq
\cX_{U_t}\cong \widetilde{h}^*\cK.\nonumber
\eeq
By \cite[Lemma 8]{Fu84}, the quotient $ K_t/\Gamma\subset M$ is the small neighborhood of $(\cY_t,\widetilde{\iota}_t^*(\beta_{\widetilde{t}_1})$ in the coarse moduli space $M$ of polarized K\"ahler manifolds, where $\Gamma$ is a finite group acts on $K_t$. Let $\sigma_t: K_t\to K_t/\Gamma$ be the quotient map. Then $g|_{U_t}=\sigma_t\circ\widetilde{g}$ and $h|_{U_t}=\sigma_t\circ\widetilde{h}$. By (\ref{equal-3}), we have $g|_{U_t}=h|_{U_t}$. Because the order $\Gamma$ is finite and  the locally complete of the moduli functor by \cite[Proposition 8]{Fu84}, there exists a holomorphic isomorphism $\chi:K_t\to K_t$ with $\chi^*\cK\cong \cK$ and
\beq
\widetilde{g}= \chi\circ\widetilde{h}: U_t\to K_t.\nonumber
\eeq
Therefore there exists a biholomorphic $\Phi:\cY_{U_t}\to \cX_{U_t}$ such that
\beq
\cY_{U_t}\cong \widetilde{g}^*\cK\cong\widetilde{h}^*\left(\chi^*\cK\right)\cong \widetilde{h}^*\cK\cong \cX_{U_t}.\nonumber
\eeq
This isomorphism commutes with the projections to $U_t$, so the families are locally isomorphic at $t$.
\end{proof}

\begin{theo}
Let $\pi:\cX\to S$ be a smooth family of projective non-unruled manifolds over a Riemann surface $S$. Suppose that $\pi':\cY\to S$ is a smooth family of compact complex manifolds satisfies that $f_t:\cY_t\to \cX_t$ is isomorphic pointwise. Then there exists an open dense subset $U\subset S$ such that for $t\in U$, the families are locally isomorphic.
\end{theo}
\begin{proof}Because the Riemann surface $S$ can be covered by countable manifold small disks, we can assume that $S=\cup_k\Delta_k$. For each $k$, consider the restricted families $\cY_{\Delta_k}:=\pi^{\prime-1}(\Delta_k)\to \Delta_k$ and $\cX_{\Delta_k}:=\pi^{-1}(\Delta_k)\to \Delta_k$, there exists an open subset $U_k$ such that the restricted families are locally isomorphic at every point of $U_k$ by Proposition \ref{main-1}. Let $U=\cup_k U_k$, then $U$ is an open dense subset and the families are locally isomorphic at every point $s\in U$.
\end{proof}

\bibliographystyle{amsplain}

\begin{thebibliography}{9}

\bibitem{ABK} Antonelli, P. L., Burghelea, D., Kahn, P. J., \textit{The non-finite homotopy type of some diffeomorphism groups}, Topology 11 (1972), 1-49.






%

\bibitem{del} del Hoyo, M., \textit{Complete connections on fiber bundles}, Indag. Math. (N.S.) 27 (2016), no. 4, 985-990.










\bibitem{Fu84}Fujiki, A., \textit{Coarse Moduli Space for Polarized Compact K\"ahler Manifolds}, Publ. Res. Inst. Math. Sci. 20 (1984), no. 5, pp. 977-1005.

\bibitem{FS90}Fujiki, A., Schumacher, G. \textit{The Moduli Space of Extremal Compact K\"ahler Manifolds and Generalized Well-Petersson Metrics}, Publ. Res. Inst. Math. Sci. 26 (1990), no. 1, pp. 101-183.





%
%
%
%




\bibitem{Kir}Kirschner, T., \textit{On the local isomorphism property for families of K3 surfaces}, arXiv:1810.11395.


\bibitem{Lesl} Leslie, J. A. \textit{On a differential structure for the group of diffeomorphisms}, Topology 6 (1967), 263-271.


%
%


\bibitem{Math-1}\textit{https://math.stackexchange.com/questions/186145/a-fiber-bundle-over-euclidean-space-is-trivial}

\bibitem{Mee} Meersseman, L., \textit{Foliated structure of the Kuranishi space and isomorphisms of
deformation families of compact complex manifolds},  Ann. Sci. $\acute{E}$c. Norm. Sup$\acute{e}$r. (4) 44.3 (2011), pp. 495-525.

%


\bibitem{Schm}Schmeding, A., \textit{An introduction to infinite-dimensional differential geometry}, Cambridge Studies in Advanced Mathematics, 202. Cambridge University Press, Cambridge, 2023.


\bibitem{Sch84}Schumacher, G.,  \textit{Moduli of polarized K\"ahler manifolds}, Math. Ann. 269, 137-144 (1984).

\bibitem{Vo} Voisin, C.,   \textit{Hodge theory and complex algebraic geometry. I. Translated from the French by Leila Schneps}, Cambridge Studies in Advanced Mathematics, 76, Cambridge University Press, Cambridge, 2007.

\bibitem{We}Wehler. J. \textit{Isomorphie von Familien kompakter komplexer Mannigfaltigkeiten}, Math. Ann. 231 (1977), pp. 77-90.

\end{thebibliography}

\end{document}